\documentclass{article}
\usepackage{verbatim}   
\usepackage{color}      
\usepackage{subfigure}  
\usepackage{hyperref}   
\usepackage{amsmath}
\usepackage{setspace}
\usepackage{algorithm,algorithmic,refcount}
\usepackage{mathtools}
\usepackage{amsfonts,colonequals,bm,nicefrac}
\usepackage{amssymb}
\usepackage[mathscr]{eucal}

\usepackage[pdftex,dvipsnames]{xcolor}  
\usepackage{xargs}                      
\usepackage[pdftex]{graphicx}
\usepackage[colorinlistoftodos,prependcaption,textsize=tiny]{todonotes}
\usepackage[body={16cm,22cm},centering]{geometry}
\usepackage{amsthm,amsfonts,mathrsfs,graphicx}
\usepackage[font=small,labelfont=bf]{caption}
\newtheorem{definition}{Definition}[section]
\newtheorem{theorem}{Theorem}[section]

\newtheorem{remark}{Remark}[section]

{\theoremstyle{definition}
}

\newtheorem{proposition}{Proposition}[section]

\newcommand{\argmin}{\arg\!\min}
\newcommand{\eqnum}{\refstepcounter{equation}\textup{\tagform@{\theequation}}}

\newcommand{\U}{{\mathscr U}}
\newcommand{\dO}{{d_\Omega}}
\newcommand{\bp}{{\bm p}}

\newcommand{\bu}{{\bm u}}
\newcommand{\bv}{{\bm v}}
\newcommand{\bx}{{\bm x}}
\newcommand{\by}{{\bm y}}
\newcommand{\bz}{{\bm z}}

\newcommand{\yrh}{{y_\mathsf{rh}}}
\newcommand{\yrhl}{{\bm y_\mathsf{rh}^\ell}}
\newcommand{\byrh}{{\bm y_\mathsf{rh}}}
\newcommand{\urh}{{\bm u_\mathsf{rh}}}
\newcommand{\urhl}{{\bm u_\mathsf{rh}^\ell}}
\newcommand{\yun}{{y_\mathsf{un}}}

\newcommandx{\info}[2][1=]{\todo[linecolor=OliveGreen,backgroundcolor=OliveGreen!25,bordercolor=OliveGreen,#1]{#2}}
\providecommand{\keywords}[1]{\textbf{\textit{Key words---}} #1}

\makeatletter
\def\namedlabel#1#2{\begingroup
    #2%
    \def\@currentlabel{#2}%
    \phantomsection\label{#1}\endgroup
}

\DeclareMathOperator{\spn}{span}

\DeclareMathOperator{\dom}{dom}
\DeclareMathOperator{\prox}{Prox}

\definecolor{DarkRed}{rgb}{0.5 0 0}
\definecolor{DarkBlue}{rgb}{0 0 0.5}
\definecolor{DarkGreen}{rgb}{0 0.5 0}
\definecolor{DarkOrange}{rgb}{0.83 0.33 0}
\definecolor{DarkMagenta}{rgb}{0.83 0 0.67}

\newcommand{\SVnote}{\textcolor{DarkGreen}}

\title{Finite-Dimensional RHC Control of Linear Time-Varying Parabolic PDEs: Stability Analysis and Model-Order Reduction}

\author{Behzad Azmi, Jan Rohleff and Stefan Volkwein}

\date{November 2023}

\begin{document}

\maketitle

\begin{abstract}
This chapter deals with the stabilization of a class of linear time-varying parabolic partial differential equations employing receding horizon control (RHC).  Here, RHC is finite-dimensional, i.e., it enters as a time-depending linear combination of finitely many indicator functions whose total supports cover only a small part of the spatial domain. Further, we consider the squared $\ell_1$-norm as the control cost. This leads to a nonsmooth infinite-horizon problem which allows a stabilizing optimal control with a low number of active actuators over time. First, the stabilizability of RHC is investigated. Then, to speed-up numerical computation, the data-driven model-order reduction (MOR) approaches are adequately incorporated within the RHC framework. Numerical experiments are also reported which illustrate the advantages of our MOR approaches.
\end{abstract}


\section{Introduction}
\label{Section:1}

Here, we are concerned with the stabilization of the control system of the form
\begin{equation}
    \label{e1}
    \left\{
    \begin{aligned}
        \dot{y}(t)-\nu\Delta y(t) + a(t)y(t)+ \nabla \cdot\big(b(t)y(t)\big)&= \sum\limits^N_{i =1} u_i(t)\mathbf{1}_{R_i}&&\text{in } (0,\infty)\times\Omega,\\
        y&=0&&\text{on } (0,\infty)\times \partial \Omega,\\
        y(0)&=y_0&&\text{on } \Omega
    \end{aligned}
    \right.
\end{equation}
with a time depending control vector $\bu\colonequals[u_1,\ldots, u_N]^\top \in \U\colonequals L^2(0,\infty;\mathbb{R}^N)$, where $\Omega \subset \mathbb{R}^\dO$ is a bounded domain with the smooth boundary $\partial \Omega$ and $\nu>0$ a diffusion parameter.  The  functions $\mathbf{1}_{R_i}$ represent the actuators. They are modeled as the characteristic functions related to open sets $R_i \subset \Omega$ for  $i=1, \dots, N$, and the support of these actuators are contained in a  \emph{small}  open subset of the domain $\Omega$. Moreover, the reaction term  $a(t)=a(t,x)$ and  convection term  $b(t)=b(t,x)$ are, respectively, real- and $\mathbb{R}^n$-valued functions of real variables $t$ and $x$. 

In this work,  we aim to derive an exponentially stabilizing control through a receding horizon control framework. To be more precise,  here the control objective is to construct an RHC  $\urh(y_0) \in  L^2(0,\infty;\mathbb{R}^N)$ for a given initial function $y_0 \in L^2(\Omega;\mathbb{R})\equalscolon L^2(\Omega)$  such that its corresponding state $\yrh$ satisfies
\begin{equation}
    \label{Eq:StEq}
    {\|\yrh(t) \|}_{L^2(\Omega)} \leq c e^{-\zeta t}\,{\| y_0\|}_{L^2(\Omega)}    \quad\text{for all }t>0   
\end{equation}
with constants $c$ and $\zeta>0$ being independent of $y_0$. 

In the RHC framework,  the current control action is obtained by minimizing a performance index defined on a finite time interval,  ranging from the current time $t_0$  to some future time $t_0+T$, with $T \in (0, \infty]$ and $t_0\in(0,\infty)$. Here, we consider the performance index function of the form
\begin{equation}
    \label{e49}
    J_{T}(\bu;t_0,y_0)\colonequals\frac{1}{2}\int_{t_0}^{t_0+T}{\|\nabla y_\bu(t)\|}^2_{L^2(\Omega;\mathbb{R}^n)}\,\mathrm dt+\frac{\beta}{2}\int^{t_0+T}_{t_0} |\bu(t)|^2_{1}\,\mathrm dt,
\end{equation}
where $y_\bu$ solves \eqref{Eq:StEq} and $\beta>0$ holds. The choice of the $\ell^1$-norm defined by $|\bu|_{1} = \sum^N_{j=1} |u_j|$ leads to a nonsmooth convex performance index function and enhances sparsity in the coefficient of the control at any $t \in (t_0,t_0+T)$. For every $t>0$ the term $|\mathbf{u}(t)|_{1}$ can be also interpreted as a convex relaxation of $|\bu(t)|_{0}$; cf., e.g., \cite{MR2243152,MR1963681,MR2045813}. That is, by incorporating the term $\nicefrac{\beta}{2}\int_{t_0}^{t_0+T}|\bu(t)|^2_{1}\,\mathrm dt$ in the performance index function,  we try to minimize the term  $\nicefrac{\beta}{2}\int_{t_0}^{t_0+T} |\bu(t)|^2_{0}\,\mathrm dt$ and, as a consequence, the number of active actuators as the time is running. Moreover, for $\bu\in\U$, we can write 
\begin{equation}
    \label{e4}
    \frac{\beta}{2}\int_{t_0}^{t_0+T} |\bu(t)|^2_{1}\,\mathrm dt=\frac{\beta}{2}\int_{t_0}^{t_0+T} |\bu(t)|^2_{2}\,\mathrm dt+\beta\int^{t_0+T}_{t_0} \sum^N_{\substack{i=1\\i<j}}|u_i(t)u_j(t)|\,\mathrm dt.
\end{equation}   
The second term on the right-hand side of \eqref{e4} is the $L^1$-penalization of the switching constraint $u_i(t)u_j(t)=0$ for $i\neq j$ and $t>0$; cf. \cite{MR3681006,MR3459600}, for instance. Then, the stabilization of the control system \eqref{e1} can also be formulated as the following infinite horizon (i.e., $T=\infty$) optimal control problem
\begin{equation}
    \label{opinf}
    \tag{$\mathbf{OP}_{\infty}(y_0)$}
    \min\big\{ J_{\infty}(\bu;0,y_0)\,\big|\,\bu\in\U\big\}\quad\text{for fixed }y_0\in L^2(\Omega).
\end{equation}
Clearly, \eqref{opinf} is a nonsmooth infinite horizon problem. For dealing with this problem, the RHC framework, known also as Model Predictive Control (MPC), offers a natural approach in which, the solution of  \ref{opinf} is approximated by solving a sequence of nonsmooth finite horizon problems which are well-studied from the theoretical and numerical aspects. These finite horizon problems have the following form. For a given initial time $\bar{t}_0$, initial functions $ \bar{y}_0 \in L^2(\Omega)$,  and prediction horizon $T$ consider the (open-loop) problem
\begin{align}
    \label{optT}
    \tag{$\mathbf{OP}_{T}(\bar{t}_0,\bar{y}_0)$}
    \min J_T(\bu;\bar{t}_0, \bar{y}_0)\quad\text{subject to (s.t.)}\quad\bu\in   \U_T(\bar t_0)\colonequals L^2(\bar t_0,\bar t_0+T;\mathbb R^N),
\end{align}
where $y=y_\bu$ solves the parabolic partial differential equation (PDE)
\begin{align}
    \label{e41}
    \left\{
    \begin{aligned}
        \dot{y}(t)-\nu\Delta y(t) + a(t)y(t)+ \nabla \cdot \big(b(t)y(t)\big)&= \sum\limits^N_{i =1} u_i(t)\mathbf{1}_{R_i} &&\text{in } (\bar{t}_0,\bar{t}_0+T)\times\Omega,\\
        y&=0&&\text{on } (\bar{t}_0,\bar{t}_0+T)\times \partial \Omega,\\
        y(\bar{t}_0)&=\bar{y}_0&&\text{on } \Omega
    \end{aligned}
    \right.
\end{align}
and $\U_T(\bar t_0)$ is supplied by the usual topology in $L^2(\bar t_0,\bar t_0+T;\mathbb R^N)$. In the receding horizon framework,  we define sampling instances $t_k\colonequals k\delta$, for $k=0,1,2,\dots$  and for a chosen sampling time $\delta >0$.  Then,  at every current sampling instance $t_k$ with state $ \yrh(t_k) \in  L^2(\Omega)$,  an  open-loop optimal control problem $(\mathbf{OP}_{T}(t_k, \yrh(t_k)))$ is solved over a finite prediction horizon $[t_k,t_k+T]$ for an appropriate prediction horizon $T>\delta$. Then, the associated optimal control is applied to steer the system from time $t_k$ with the initial state $ \yrh(t_k) \in  L^2(\Omega)$ until time $t_{k+1}\colonequals t_k+\delta$ at which point, a new measurement of the state $\yrh(t_{k+1}) \in  L^2(\Omega)$, is assumed to be available. The process is repeated starting from this new measured state:  we obtain a new optimal control and a new predicted state trajectory by shifting the prediction horizon forward in time. The sampling time $\delta$ is the period between two sample instances. Throughout, we denote the receding horizon state- and control variables by $\yrh(\cdot)$ and $\urh(\cdot)$, respectively. Also,  $(y_T^*(\cdot\,;\bar{t}_0, \bar y_0), \bu^*_T(\cdot\,; \bar t_0, \bar{y}_0))$ stands for the optimal state and control of the optimal control problem with finite time horizon $T$,  and initial function  $\bar y_0$ at initial time $\bar t_0$. This is summarized in Algorithm~\ref{RHA}.

This manuscript reviews the receding horizon framework proposed for linear time-varying parabolic equations in \cite{AK19} with finitely many controllers and the corresponding stability and suboptimality results. To guarantee the stability of RHC, neither terminal costs nor terminal constraints are needed instead, by generating an appropriate sequence of overlapping temporal intervals and applying a suitable concatenation scheme, the stability and suboptimality of RHC are obtained. Previously, this framework was studied for time-invariant continuous-time finite-dimensional controlled systems in, e.g., \cite{MR1833035,MR2950434}, and for time-invariant discrete-time controlled systems in, for instance, \cite{MR2141559,MR2491596,MR2459581}. Then, by incorporating a model-order reduction (MOR) technique, we show that the computation of RHC can be significantly speed-up with impressive accuracy.

\begin{algorithm}[htbp]
\caption{RHC($\delta,T$)}\label{RHA}
\begin{algorithmic}[1]
\REQUIRE{The sampling time $\delta$, the prediction horizon $T\geq \delta$, and the initial state $y_0$;}
\ENSURE{The stability of RHC~$\mathbf{u}_{rh}$.}
\STATE Set~$(\bar{t}_0,\bar{y}_0)\colonequals(0,y_0)$ and $\yrh(0)=y_0 $;
\STATE Find the the optimal solution $(\bar y_T(\cdot\,;\bar{t}_0,\bar{y}_0),\bar{\bu}_T(\cdot\,;\bar{t}_0,\bar{y}_0))$
over the time horizon $(\bar{t}_0,\bar{t}_0+T)$ by solving the open-loop problem~\eqref{optT};
\STATE For all $\tau\in[\bar{t}_0,\bar{t}_0+\delta)$ set $\yrh(\tau)\colonequals\bar y_T(\tau;\bar{t}_0,\bar{y}_0)$ and $\urh(\tau)\colonequals\bar\bu_T(\tau;\bar{t}_0,\bar{y}_0)$;
\STATE Update: $(\bar{t}_0,\bar{y}_0)\leftarrow(\bar{t}_0+\delta,\yrh(\bar{t}_0+\delta;\bar{t}_0,\bar{y}_0))$;
\STATE Go to Step  2;
\end{algorithmic}
\end{algorithm}

In a classical approach, \eqref{e1} (state equation) is approximated by a high-dimensional full-order model (FOM) resulting from discretization. For the spatial discretization, a finite element (FE) method is often used, leading to high-dimensional dynamical systems. Hence, the complexity of the optimization problem directly depends on the number of degrees of freedom (DOF) of the FOM. Mesh adaptivity has been advised to minimize the number of DOFs; see, e.g., \cite{BKR00,LY01}. A more recent approach is the usage of model-order reduction (MOR) methods to replace the FOM with a surrogate reduced-order model (ROM) of possibly very low dimension. MOR is a highly active research field that has seen tremendous development in recent years, both from a theoretical and application point of view. For an introduction and overview, we refer to the monographs and collections \cite{BOCW17,HRS16,HLBR13,QMN16}, for instance. In the context of optimal control, ROM is utilized, e.g., in \cite{Ant09,BBH18,BMS06,GV17,HJ18,HV08,NMT11,SBMR18,TV09}. In MPC, ROM is applied in, e.g., \cite{AV15,GU14,HWG06,LMFP22,LRHYA14,Mec19,MV19,Roh23}.

Here, our MOR approach is based on proper orthogonal decomposition (POD), and it proceeds through the following steps. First, we compute the RHC on the first temporal interval $(0, T)$ for the FOM obtained by FE discretization for both state and adjoint equations, and store them all as the FE snapshots. Next, we generate a POD basis for both the state and adjoint equations from these snapshots. Finally, we compute the RHC for the subsequent intervals, using the associated surrogate ROM for both state and adjoint equations.

 We should also note that while MOR techniques have been explored in the context of MPC, to the best of our knowledge, there are very few studies addressing MPC incorporated with MOR for time-varying control systems and sparsity-promoting control costs.

The remaining sections of this paper are structured as follows: Section~\ref{Section:2} commences by introducing the notation used throughout the paper and revisiting the results concerning the well-posedness of the state and equation and open loop problems within Algorithm~\ref{RHA}. Section~\ref{Section:3} reviews the assumptions and results related to the stabilizability of the control system employing RHC. In Section~\ref{Section:4}, we present the details of the POD-based RHC. Finally, in Section~\ref{Sec:NE}, we report numerical experiments demonstrating the efficiency of the proposed POD-based RHC.

\section{Well-posedness}
\label{Section:2}

In this section, we review some preliminaries about the well-posedness of the state equation \eqref{e1} and the finite horizon optimal control problems \eqref{optT}  within Algorithm \ref{RHA}. Beforehand, we introduce the following 
function spaces. We set $H\colonequals L^2(\Omega ; \mathbb{R})$, $V\colonequals H^1_0(\Omega; \mathbb{R})$, and $V':=H^{-1}(\Omega;\mathbb{R})$, and endow $V$ by the following inner product and corresponding norm
\begin{equation*}
(\phi,\psi)_V\colonequals(\nabla \phi ,\nabla \psi)_{H^\dO},\quad{\|\phi\|}_V\colonequals(\phi,\phi)^{1/2}_V={\|\nabla \phi \|}_{H^\dO} \quad\text{for every  }  \phi, \psi \in V
\end{equation*}
with
\begin{align*}
    H^\dO=\underbrace{H\times\ldots\times H}_{\dO\text{-times}}.
\end{align*}
By identifying  $H$ with its dual, we obtain a Gelfand triple $V\hookrightarrow H \hookrightarrow V'$ of separable Hilbert spaces with dense injections. Finally, for every open interval $(s_1, s_2)\subset[0,\infty)$, we can define the space $W(s_1,s_2)$ by
\begin{align*}
    W(s_1,s_2)\colonequals\big\{v\in L^2(s_1,s_2;V):\dot v \in L^2(s_1,s_2; V')\big\}
\end{align*}
endowed with the norm%
\begin{align*}
    {\|v\|}_{W(s_1,s_2)}\colonequals\left({\|v\|}^2_{L^2(s_1,s_2;V)}+{\|\dot v\|}^2_{L^2(s_1,s_2; V')}\right)^{1/2}.
\end{align*}
Here and throughout, $\dot v$ denotes the distributional derivative for any $v\in W(s_1,s_2)$ with respect to $t$. It is well-known that $W(s_1,s_2) \hookrightarrow C([s_1,s_2];H)$; cf., e.g., \cite[Theorem 3.1]{MR0350177}. 

We define
\begin{equation}
    \label{e93a}
    U_\omega\colonequals\big\{\mathbf{1}_{R_i}\,\big|\,i=1,\dots,N\big\}\quad\text{with }\bigcup^N_{i=1} R_i \subset \omega \subseteq \Omega.
\end{equation}
Now we introduce the linear and bounded operator $\mathcal B_{U_\omega}:\mathbb R^N\to H$ by
\begin{align*}
    \mathcal \mathbb R^N\ni\bv=[v_1,\ldots,v_N]^\top\mapsto \mathcal B\bv=\sum_{i=1}^Nv_i\mathbf 1_{R_i}\in H.
\end{align*}
To study the well-posedness of the state equation on the finite horizon $[\bar t_0,\bar t_0+T]\subsetneq[0,\infty)$ we consider
\begin{equation}
    \label{e17}
    \begin{aligned}
        \dot{y}(t)-\nu\Delta y(t) + a(t)y(t)+ \nabla \cdot \big(b(t)y(t)\big)&=\mathcal B\big(\bu(t)\big)&&\text{in } (\bar{t}_0,\bar{t}_0+T)\times\Omega
,\\
        y&=0&& \text{on }(\bar{t}_0,\bar{t}_0+T)\times \partial \Omega,\\
        y(\bar{t}_0)&=\bar{y}_0&&\text{on } \Omega.
    \end{aligned}
\end{equation}

Throughout the manuscript, we assume that
\begin{equation}
    \label{e56}
    \tag{RA}
    a \in L^{\infty}(0, \infty;  L^r(\Omega)) \text{ with } r\geq\dO \colonequals\dim(\Omega)\text{ and }b \in L^{\infty}((0,\infty)\times \Omega ; \mathbb{R}^\dO).
\end{equation}

\begin{remark}
    It follows from \eqref{e56} that $(a(\cdot)y(\cdot),\varphi)_H$ is essentially bounded on any finite interval $[s_1,s_2]\subset[0,\infty)$ for every $y\in W(s_1,s_2)$ and $\varphi\in V$. 
\end{remark}
 
We also recall the following notion of weak variational solution for \eqref{e17}.

\begin{definition}
    Let  $(\bar{t}_0,T)\in\mathbb R^2_+$ and $(\bar{y}_0,\bu)\in H \times \U_T(\bar t_0)$ be given.  Then,  a function $y \in W(\bar{t}_0,\bar{t}_0+T)$ is referred to as a weak solution of \eqref{e17} if for almost every $t \in (\bar{t}_0, \bar{t}_0+T)$ we have
    \begin{equation}
        \label{e19}
        {\langle\dot{y}(t),\varphi\rangle}_{V',V}+\nu\,{(y(t),\varphi)}_V+{(a(t)y(t),\varphi)}_H-{(b(t)y(t),\nabla\varphi)}_{H^\dO}={(\mathcal B(\bu(t)),\varphi)}_H
    \end{equation}
    for all $\varphi \in V$,  and $y(\bar{t}_0)= \bar{y}_0$ is satisfied in $H$.
\end{definition}

For this weak solution, we have the following existence results and energy estimates.

\begin{proposition}
    \label{Theo2}
    For every $(\bar{t}_0,T)\in\mathbb R^2_+$ and $(\bar{y}_0,
    \bu)\in H \times \U_T(\bar t_0)$, equation \eqref{e17} admits a unique weak solution $y \in W(\bar{t}_0,\bar{t}_0+T)$ satisfying
    \begin{equation}
         {\|y\|}^2_{C([\bar{t}_0,\bar{t}_0+T];H)}+{\| y \|}^2_{W(\bar{t}_0,\bar{t}_0+T)}\leq c_1\left({\|\bar{y}_0\|}^2_H+{\|\bu\|}^2_{\U_T(\bar t_0)}\right),\\ \label{e13}
    \end{equation}
    with  $c_1$ depending on  $(T,\nu,a,b,\Omega)$.
\end{proposition}

\begin{proof}
    The proof is given in \cite{AK19}. 
\end{proof}

In Algorithm \ref{RHA} the finite horizon optimal control problems of the form \eqref{optT} have to be solved repeatedly in step~2. We next investigate these optimal control problems. In this matter, we express the cost functional in \eqref{optT} as
\begin{equation}
    \label{e94}
    J_{T}(\bu;\bar{t}_0,\bar{y}_0)=\mathcal{F}^{\bar{t}_0,\bar{y}_0}_T(\bu)+\mathcal{G}_T^{\bar{t}_0}(\bu)\quad\text{for }\bu\in\U_T(\bar t_0),
\end{equation}
where
\begin{align*}
    \mathcal{F}_T^{\bar{t}_0,\bar{y}_0}(\bu)\colonequals\frac{1}{2}\,\big\|\mathcal L_T^{\bar{t}_0,\bar{y}_0} \bu\big\|^2_{L^2(\bar{t}_0,\bar{t}_0+T;V)},\quad\mathcal G_T^{\bar{t}_0}(\bu)\colonequals\frac{\beta}{2}\int^{\bar{t}_0+T}_{\bar{t}_0}| \bu(t)|^2_1 \,\mathrm dt
\end{align*}
and $\mathcal L_T^{\bar{t}_0,\bar{y}_0}:\U_T(\bar t_0)\to W(\bar t_0,\bar t_0+T)$ stands for the control-to-state operator for \eqref{e41}. From Proposition \ref{Theo2}, it follows that $\mathcal{F}^{\bar{t}_0,\bar{y}_0}_T:\U_T(\bar t_0)\to\mathbb R$ is well-defined, convex, and continuously Fr\'echet differentiable. Further, $\mathcal{G}_T^{\bar{t}_0}:\U_T(\bar t_0)\to\mathbb R$ is a proper nonsmooth convex function. Hence,  the nonnegative objective function $J_{T}(\cdot\,;\bar{t}_0,\bar{y}_0):\U_T(\bar t_0)\to\mathbb R$ is weakly lower semi-continuous and radially unbounded, and thus, the existence of a unique minimizer for \eqref{optT} is established through the direct method in the calculus of variations; cf., e.g., in \cite{MR2361288}. Uniqueness is ensured by the strict convexity of $\mathcal{F}^{\bar{t}_0,\bar{y}_0}_T$ which is justified by the injectivity of $\mathcal L_T^{\bar{t}_0,\bar{y}_0}$.

\begin{proposition}
    \label{prop1}
    For every $(\bar{t}_0,T)\in \mathbb{R}_+^2$ and $\bar{y}_0\in H$ the finite horizon problem \eqref{optT} admits a unique minimizer $\bu^*$.
\end{proposition}

We now proceed to derive the first-order optimality condition for \eqref{optT}. Given that $\mathcal F^{\bar t_0,\bar y_0}_T$ is differentiable and $\dom(\mathcal G^{\bar t_0}) =\U_T(\bar t_0)$, the first-order optimality condition for the minimizer $\bu^*$ can be formulated as follows (cf., e.g., \cite{Cla90})
\begin{equation}
    \label{e51}
    0\in\partial\big(\mathcal F^{\bar t_0,\bar y_0}_T+\mathcal G^{\bar t_0}_T\big)(\bu^*)=\partial\mathcal F^{\bar t_0,\bar y_0}_T(\bu^*)+\partial\mathcal G^{\bar t_0}_T(\bu^*)=\big\{ D\mathcal F^{\bar t_0,\bar y_0}_T(\bu^*)\big\}+\partial\mathcal G^{\bar t_0}_T(\bu^*),
\end{equation}
where $D\mathcal F^{\bar t_0,\bar y_0}_T$ is the first Fr\'echet derivative of $\mathcal F^{\bar t_0,\bar y_0}_T$ with respect to $\bu$. We introduce the following adjoint equation,
\begin{equation}
    \label{e52}
    \left\{
    \begin{aligned}
        -\dot{p}(t)-\nu \Delta p(t)+ a(t)p(t)-\big(b(t) \cdot \nabla p(t)\big)&= -\Delta y^*(t)    &&\text{in }(\bar{t}_0,\bar{t}_0+T)\times\Omega,\\
        p&=0&&\text{on }(\bar{t}_0,\bar{t}_0+T)\times\partial\Omega,\\
        p(\bar{t}_0+T)&=0&&\text{on }\Omega
    \end{aligned}
    \right.
\end{equation}
with $y^*=y(\bu^*)\in W(\bar{t}_0,\bar{t}_0+T)$ as the solution of \eqref{e41} for $\bu=\bu^* \in\U_T(\bar t_0)$. Then  $D\mathcal F^{\bar t_0,\bar y_0}_T$ can be expressed as $D\mathcal F^{\bar t_0,\bar y_0}_T(\bu^*)=\mathcal B^\star p$ in  $\U_T(\bar{t}_0)$, where $\mathcal B^\star:H\to\mathbb R^N$ stands for the adjoint operator of $\mathcal B$ given as

\begin{align*}
    \mathcal B^\star\varphi=\big[{(\mathbf 1_{R_1},\varphi)}_H\,|\ldots|\,{(\mathbf 1_{R_N},\varphi)}_H\big]^\top\in\mathbb R^N\quad\text{for }\varphi\in H.
\end{align*}
Now, the optimality condition \eqref{e51} can be stated as
\begin{equation}
    \label{e55}
    -\mathcal B^\star p \in \partial \mathcal G_T^{\bar t_0}(\bu^*),
\end{equation}
where $p=p(y^*)$ is the weak solution to \eqref{e52}. Well-posedness of the adjoint equation follows with similar arguments given in  \cite{AK19} and the fact that  $\Delta y^* \in L^2(\bar{t}_0, \bar{t}_0+T;V')$ holds true
.

To numerically address sub-problems of the form \eqref{e94}, commonly employed approaches are the forward-backward splitting (FBS) algorithms. These methods rely on iteratively evaluating the proximal operator $\prox_{\mathcal{G}}(\hat{\bu}): L^2(\bar{t}_0, \bar{t}_0+T; \mathbb{R}^N) \to L^2(\bar{t}_0, \bar{t}_0+T; \mathbb{R}^N)$ defined as:
\begin{equation*}
    \prox_{\mathcal{G}}(\hat{\bu})\colonequals\argmin\bigg\{\frac{1}{2}\,{\|\bu-\hat{\bu}\|}^2_{L^2(\bar t_0,\bar t_0+T;\mathbb{R}^N)} +\mathcal G_T^{\bar t_0}(\bu):\bu\in\U_T(\bar t_0)\bigg\}.
\end{equation*}
The well-posedness of $\prox_{\mathcal{G}}$ is justified by the properties that $\mathcal G_T^{\bar t_0}$ is proper, convex, and weakly lower semi-continuous. Subsequently, the following proposition articulates the first-order optimality conditions in terms of the proximal operator. This optimality condition serves as the termination criterion for the FBS algorithm that will be utilized later.

\begin{proposition}
    Let $(T,\bar{t}_0)\in\mathbb{R}_+^2$ and $\bar{y}_0\in H$ be given. Then $\bu^*\in\U_T(\bar t_0)$ is the unique minimizer to \eqref{optT} iff there exists a \SVnote{(dual)} solution $p^*=p(y^*(\bu^*)) \in W(\bar{t}_0,\bar{t}_0+T)$ to \eqref{e41} such that  the following equality holds
    \begin{equation}
        \label{e54}
        \bu^*=\prox_{\bar{\alpha}\mathcal{G}}\big(\bu^* -\bar\alpha\mathcal B^\star p^*\big)\quad\text{for any }\bar\alpha>0. 
    \end{equation}
    Here $y^*(\bu^*)$ is the solution to \eqref{e41} for $\bu=\bu^*$.
\end{proposition}

\begin{proof}
    One needs only to verify the equivalence between the inequalities \eqref{e55} and \eqref{e54} which is done as in  \cite[Corollary 27.3]{MR2798533}.
\end{proof}

Due to \cite[Proposition 16.63]{MR2798533}, the subdifferential  of $\mathcal G_T^{\bar t_0}$ at $\bu$ is characterized by
\begin{equation*}
    \partial \mathcal G_T^{\bar t_0}(\bu)\colonequals\left\lbrace \bv \in\U_T(\bar t_0):\bv(t) \in  \partial\Big( \frac{\beta}{2}{|\bu(t)|}^2_{1}\Big)  \text{ for a.e. } t\in (\bar{t}_0,\bar{t}_0+T)  \right\rbrace. 
\end{equation*}
Here, we provide the pointwise characterization of $\prox_{\bar{\alpha}\mathcal{G}}$, a fundamental aspect for the FBS algorithm to be employed later in Section \ref{Sec:NE}. Using \cite[Proposition 24.13]{MR2798533} and setting $g\colonequals\nicefrac{\beta}{2}\,| \cdot|^2_{1}$, the proximal operator $\prox_{\bar{\alpha}\mathcal{G}}$ for any $\bar{\alpha}>0$ can be expressed pointwise as
\begin{equation}
    \label{e87a}
    \left[\prox_{\bar{\alpha}\mathcal{G}}(\bu)\right] (t) = \prox_{\bar{\alpha} g}(\bu(t))\quad\text{for almost all } t\in (\bar{t}_0 ,\bar{t}_0+T),
\end{equation}
and thus the first-order optimality conditions \eqref{e54} can be stated as
\begin{equation}
    \label{e87}
    \bu^*(t)=\prox_{\bar{\alpha} g}(\bu^*(t)-\bar{\alpha} \mathcal B^\star p(t))\quad\text{for almost all } t\in (\bar{t}_0,\bar{t}_0+T).
\end{equation}

Thus, the remaining task involves computing the proximal operator of $\bar{\alpha} g: \mathbb{R}^N \to \mathbb{R}_+$. By following the same argument as in \cite[Lemma 6.70]{MR3719240} and \cite{MR3380641}, one can establish for any $\bx:=(x_1,\dots,x_N)^\top \in \mathbb R^N$ that
\begin{equation}
    \label{e85}
    \prox_{\bar{\alpha} g }(\bx)\colonequals
    \left\{
    \begin{aligned}
        &\left(\frac{\lambda_ix_i}{\lambda_i+\bar{\alpha} \beta} \right)^N_{i =1}&&\text{if }\bx\neq 0,\\
        &0&&\text{if }\bx= 0,
    \end{aligned}\right\}\quad\text{for any }\bx=[x_1,\dots,x_N]^\top \in \mathbb R^N
\end{equation}
where we set
\begin{align*}
    \lambda_i\colonequals\Bigg[\frac{\sqrt{\frac{\bar{\alpha} \beta}{2}}|x_i| }{\sqrt{\mu^*}}-\bar{\alpha}\beta\Bigg]_+\quad\text{with }[\cdot]_+\colonequals\max(0, \cdot )
\end{align*}
and $\mu^*$ being any positive zero of the following one-dimensional nonincreasing function
\begin{equation}
    \label{e95}
    \psi(\mu)\colonequals\sum^N_{i=1} \left[ \frac{\sqrt{\frac{\bar{\alpha} \beta}{2}}|x_i| }{\sqrt{\mu}}-\bar{\alpha} \beta \right]_+-1.
\end{equation}
In other words, $\mu^*$  is chosen so that  $\sum^N_{i=1} \lambda_i = 1$.

\begin{remark}
    Due to the characterization \eqref{e85} of $\prox_{\bar{\alpha} g}$, the cardinality of the set
    \begin{align*}
        \mathscr D(\bx)\colonequals\Bigg\{i\in \{1,\dots,N \}:\frac{\sqrt{\frac{\bar{\alpha}\beta}{2}}|x_i|}{\sqrt{\mu^*}}-\bar{\alpha}\beta > 0\Bigg\}    
    \end{align*}
    is the number of nonzero components  of $\prox_{\bar{\alpha} g}(\bx)$.  Hence, due to \eqref{e87},   $|\mathscr D^*(t)|\colonequals|\mathscr D(\bu^*(t)-\mathcal B^\star p(t))| $ stands for the number of nonzero components of  $\bu^*$ at time $t$.
\end{remark}

\section{Stability of RHC}
\label{Section:3}

In this section,  we review selected results on the stabilization of \eqref{e1} by finitely many controllers computed by Algorithm \ref{RHA}. Beforehand, we will define the value functions. 

\begin{definition}
    For any $y_0 \in H$  the infinite horizon value function $V_{\infty}: H \to \mathbb{R}_+$ is defined by
    \begin{equation*}
        V_{\infty}(y_0)\colonequals\min\big\{J_{\infty}(\bu;0,y_0): \bu\in\U\text{ satisfies \eqref{e1}}\big\}.
    \end{equation*}
    Similarly, for every $(\bar t_0,T)\in \mathbb{R}^2_+$ and $\bar{y}_0\in H$ the finite horizon value function $V_T: \mathbb{R}_+ \times H \to \mathbb{R}_+$ is defined by
    \begin{equation*}
        V_{T}(\bar{t}_0,\bar{y}_0)\colonequals\min\big\{J_T(\bu;\bar{t}_0,\bar{y}_0):\bu\in\U_T(\bar t_0)\text{ satisfies \eqref{e41}}\big\}.
    \end{equation*}
\end{definition}

The suboptimality of the RHC computed by Algorithm \ref{RHA} is expressed in terms of the infinite horizon value function, while the finite horizon value function serves as the Lyapunov function to establish stabilizability. To establish $V_{T}$ as a Lyapunov function, it is essential to demonstrate its uniform decrement and positivity with respect to the $H$-norm. The former is justified by the stabilizability property of the control system \eqref{e1}. Following arguments similar to those presented in \cite{MR3691212,7330942,MR3337988,Phan2018} for time-varying controlled systems, it can be shown that for $\lambda>0$ and the set of actuators $U_\omega $ defined in \eqref{e93a} and $\Pi_N:H \to\spn{(U_\omega)} \subset H\hookrightarrow V'$ as the orthogonal projection onto $\spn{(U_\omega)}$ in $H$, if the condition
\begin{equation}    
    \label{e21}
    \tag{coac}
    {\|\mathcal E_{H,V'}-\Pi_{N}\|}^2_{\mathcal{L}(H,V')}<\Upsilon^{-1},
\end{equation}
holds for a constant $\Upsilon\colonequals\Upsilon(\lambda,a,b)>0$
and the canonical embedding operator $\mathcal E_{H,V'}:H\to V'$, then there exists a stabilizing feed back control $\hat\bu=\hat\bu(y_0)$ with $\hat\bu(t)=[\hat u_1(t),\ldots,\hat u_N(t)]^\top$ which steers the system \eqref{e1} to zero exponentially with rate $\lambda$; cf., e.g., \cite[Theorem 2.10]{MR3691212}.

As an illustrative example, consider the case where $\omega$ is an open rectangle defined as
\begin{equation}
    \label{e100}
    \omega\colonequals\prod^n_{i=1}(l_i,u_i)\subset \Omega.
\end{equation}
We consider the uniform partitioning of $\omega$ into a family of sub-rectangles. For each $i=1,\dots,n$, the interval $(l_i,u_i)$ is subdivided into $d_i$ intervals denoted by $I_{i,k_i}=(l_i+k_i\delta_i,l_i+(k_i+1)\delta_i)$ with $\delta_i=\nicefrac{(u_i-l_i)}{d_i}$, $k_i\in {0,1,\dots,d_i-1}$, and $i=1,\ldots,n$. Consequently, $\omega$ is partitioned into $N\colonequals\prod^n_{i=1} d_i$ sub-rectangles defined as
\begin{equation}
    \label{e92}
    \big\{R_i : i \in \{ 1,\dots,N\}\big\}\colonequals\bigg\{\prod^n_{i =1}I_{i,k_i} : k_i \in \{ 0,1,\dots, d_i  -1\}\bigg\}.
\end{equation}
For this selection of actuators, it was demonstrated in \cite[Example 2.12]{MR3691212} and \cite[Section IV]{7330942} that \eqref{e21} is satisfied, provided that $N \geq( (\nicefrac{\bar{I}^2}{\pi^2})\Upsilon)^{\nicefrac{n}{2}}$, where $\bar{I}:= \max_{1\leq i \leq N} (u_i-l_i)$. This relationship provides a lower bound on the number of actuators, ensuring exponential stabilizability. It is defined with respect to the chosen parameters $\lambda$, $a$, $b$, $\nu$, and the set of actuators defined by \eqref{e93a} and \eqref{e92}, with this dependence expressed in terms of the values of $\Upsilon(\lambda,a,b,U_\omega)$ and $\bar{I}$.

In the next theorem, we state the exponential stability of RHC obtained by Algorithm \ref{RHA}. 

\begin{theorem}[Suboptimality and exponential decay]
    \label{subopth}
    Assume that for $U_\omega\subset H$ given in \eqref{e93a} and $\lambda>0$, condition \eqref{e21} is satisfied with a real number $\Upsilon>0$.   Then, for any given $\delta$,  there exist numbers $T^*=T^*(\delta,U_\omega)>\delta$ and $ \alpha = \alpha(\delta,\mathcal{U}_{\omega}) \in (0,1)$ such that for every prediction horizon $T\geq T^*$, the RHC $\urh\in L^2(0,\infty;\mathbb{R}^N)$ obtained by Algorithm~{\em\ref{RHA}} is globally \textbf{suboptimal}  and \textbf{exponentially stable}. That is, it satisfies 
    \begin{equation}
        \label{ed27}
        \alpha V_{\infty}(y_0)\leq\alpha J_{\infty}(\urh;0,y_0)\leq V_T(0,y_0) \leq V_{\infty}(y_0)
    \end{equation}
    and 
    \begin{equation}
        \label{ed28}
        {\|\yrh(t)\|}^2_{H} \leq c_He^{-\zeta t}\,{\|y_0\|}^2_H\quad\text{for }t\geq 0
    \end{equation}
    for every $y_0\in H$, where the positive numbers $\zeta$ and $c_H$  depend on $\alpha$, $\delta$ and $T$, but are independent of $y_0$. 
\end{theorem}

\begin{proof}
    For the proof we refer to \cite{AK19}.
\end{proof}

\begin{remark}
    For fixed $\delta >0$, it can be shown that  $\lim_{T \to \infty} \alpha(T) = 1$. Thus RHC is asymptotically optimal. Moreover, for fixed $T\geq T^*$ we obtain  that $\alpha \to -\infty$ as $\delta \to 0$. That is,  for arbitrarily small sampling times  $\delta$,  the suboptimality and asymptotic stability of RHC are not guaranteed.
\end{remark}

\section{Model-Order Reduction (MOR)}
\label{Section:4}

In this section, following a brief introduction to MOR, we introduce the POD-based RHC algorithm. MOR is a technique that addresses the challenge of simplifying mathematical models without substantially compromising their accuracy. The primary objective of MOR is to achieve model simplification while preserving essential characteristics. In our specific case, the objective is to utilize the ROM for constrained dynamics, coupled with RHC, to identify a sequence of optimal controllers that stabilize the reduced dynamic. Since the ROM captures the core features of the full dynamical system, the resultant control also functions as an approximation for RHC (the sequence of optimal controllers for the full-order model system) to stabilize the entire system. The main advantage of MOR lies in its ability to significantly reduce the computational resources and complexity involved in simulating and controlling large-scale systems. A satisfactory ROM should exhibit the following characteristics:
\begin{itemize}
    \item[(i)] The approximation error should be small.
    \item[(ii)] The ROM should be computationally efficient.
    \item[(iii)] The procedure to get the ROM should be automatic and also computationally efficient.
\end{itemize}

For our example, we introduce two low-dimensional subspaces $V^\ell_y \subset V$ and $ V^\ell_p \subset V$ for the state equation \eqref{e17} and the adjoint equation \eqref{e52}, respectively. Then for given $\bu\in\U_T^{\bar t_0}$ the reduced-order state $y^\ell=y^\ell_\bu\in W(\bar t_0,\bar t_0+T)$ satisfies for almost all $t\in(\bar t_0, \bar t_0+T)$
\begin{align}
    \label{ROMstate}
    \left\{
    \begin{aligned}
        {\langle\dot y^\ell(t),\psi\rangle}_{V',V}+\nu\,{(\nabla y^\ell(t),\nabla \psi)}_H+{(a(t)y^\ell(t),\psi)}_H-{(b(t)y^\ell(t),\nabla \psi)}_H&=\sum_{i=1}^Nu_i(t){(\mathbf{1}_{R_i},\psi)}_H,\\
        y^\ell(\bar{t}_0)&=\mathcal P^\ell\bar{y}_0\quad\text{in } H
    \end{aligned}
    \right.
\end{align} 
for all $\psi \in V^\ell_y$, where $\mathcal P^\ell:H \to V^\ell_y$ is a linear and bounded (projection) operator. Further, for the corresponding reduced-order dual variable $p^\ell=p^\ell_\bu\in W(\bar t_0,\bar t_0+T)$ it holds for almost all $t \in (\bar t_0,\bar t_0+T)$ and all $\psi\in V^\ell_p$ that
\begin{align*}
    \left\{
    \begin{aligned}
        -{\langle\dot p^\ell(t),\psi\rangle}_{V',V}+\nu\,{(\nabla p^\ell(t),\nabla\psi)}_H+{(a(t)\psi,p^\ell(t))}_H-{(b(t)\varphi,\nabla p^\ell(t))}_H&={(y^\ell(t),\varphi)}_V,\\
        p^\ell(\bar{t}_0+T)&= 0 \quad \text{in } H,
    \end{aligned}
    \right.
\end{align*}
where  $y^\ell$ is the solution of \eqref{ROMstate}.

To obtain a computationally ROM, we introduce a finite element (FE) discretization of the state equation \eqref{e17} and its corresponding adjoint equation \eqref{e52}.
Providing further details, let $\varphi_1, \ldots, \varphi_m \in V$ represent the FE ansatz functions. The FE subspace, denoted as $V^\mathsf{fe}$, is defined by $V^\mathsf{fe} = \spn \{ {\varphi_1, \ldots, \varphi_m}\} \subset V$. For $(t, x) \in (0, \infty) \times \Omega$, we approximate the state as
\begin{equation*}
    y(t,x)\approx y^\mathsf{fe}(t,x)=\sum_{i=1}^m y_i^\mathsf{fe}(t)\varphi_i(x)\quad\text{for }  (t,x)\in (0,\infty)\times \Omega,
\end{equation*}
with initial value
\begin{equation*}
    y_0(x)\approx \sum_{i=1}^m y_{0i}^\mathsf{fe}\varphi_i(x)\quad\text{for }  x\in \Omega.
\end{equation*}
We also introduce the $(m\times m)$-matrices
\begin{align*}
    M_{ij}&={(\varphi_j,\varphi_i)}_H, &S_{ij}&={(\varphi_j,\varphi_i)}_V &&\text{for } 1\leq i,j \leq m,\\
    C(t)_{ij}&={(\varphi_j,a(t)\varphi_i)}_H,&D(t)_{ij}&={(\varphi_j,\nabla\cdot(b(t)\varphi_i))}_H &&\text{for } 1\leq i,j \leq m, \\
    A(t)&= \nu S - C(t)-D(t)
\end{align*}
and the $(m \times N)$-FE-control matrix
\begin{align*}
    B_j= \left[{(\mathbf{1}_{R_1},\varphi_j)}_H, \dots, 
{(\mathbf{1}_{R_N},\varphi_j)}_H\right] \quad \text{for } 1\leq j \leq m.
\end{align*}

For the sake of convenience, we will proceed in the remaining part of the section by defining
\begin{align*}
    \by(t)=\big[y_1^\mathsf{fe}(t),\ldots,y_m^\mathsf{fe}(t)\big]^\top\in\mathbb R^m\quad\text{and}\quad\by_0= \big[y_{0,1}^\mathsf{fe},\ldots,y_{0,m}^\mathsf{fe}\big]^\top\in\mathbb R^m.    
\end{align*}
Now, problem \eqref{optT} is approximated by the following FE optimization problem
\begin{equation}
    \label{eq:opt_discrete}
    \tag{$\mathbf{OP}^{\mathsf{fe}}_T(\bar{t}_0,\bar{\by}_0)$}
    \begin{aligned}
        &\min J_T^\mathsf{fe}(\bu;\bar{t}_0,\bar{\by}_0)=\frac{1}{2}\int_{\bar{t}_0}^{\bar{t}_0+T}\by(t)^\top S\by(t)+\beta\,{|\bu(t)|}^2_1\,\mathrm dt\\
	    &\hspace{0.5mm}\text{s.t. }\bu(t) \in \mathbb R^N \text{ and }\by(t )\in\mathscr Y^\mathsf{fe}\text{ solves the initial value problem}\\
	    &\hspace{12mm}M\dot\by(t)=A(t)\by(t)+B\bu(t)\text{ for }t\in(\bar{t}_0,\bar{t}_0+T),\quad\by(0)=\bar{\by}_0
    \end{aligned}
\end{equation}
with $\mathscr Y^\mathsf{fe}=H^1((\bar{t}_0,\bar{t}_0+T);\mathbb R^m)$. 
Later on, we will use the notation
\begin{equation}
    \label{e25a}
    \begin{aligned}
        \dot\by(t)=\bm f_y(t,\by(t),\bu(t))\text{ for } t\in(\bar t_0,\bar t_0+T),\quad\by(\bar t_0)=\bar{\by}_0,
    \end{aligned}
\end{equation}
where the state dynamic $\bm f_y$ is given by
\begin{align*}
   \bm f_y(t,\by(t),\bu(t))\colonequals M^{-1}(A(t)\by(t)+B\bu(t)).
\end{align*}
The discrete first-order sufficient optimality condition reads as follows:
\begin{equation*}
    0=\nabla J_T^\mathsf{fe}(\bu,\bar{t}_0,\bar{\by}_0)= \beta \bu(t) - B^\top \bp(t) \quad \text{for a.e. }  t\in (\bar{t}_0,\bar{t}_0+T),
    \end{equation*}
where $ \bp(t)\colonequals[p_1^\mathsf{fe}(t),\ldots,p_m^\mathsf{fe}(t)]^\top\in\mathbb R^m$ is the solution of the discrete adjoint equation 
\begin{equation}
    \label{eAdj}
    \begin{aligned}
        -\dot\bp(t)=\bm f_p(t,\bp(t),\by(t))\text{ for } t\in(\bar t_0,\bar t_0+T),\quad\bp(\bar{t}_0+T)=0,
    \end{aligned}
\end{equation}
and the adjoint dynamic $\bm f_p$ is given by
\begin{align*}
   \bm f_p(t,\bp(t),\by(t))\colonequals M^{-1}\big(A^\top(t)\by(t)-S\by(t)\big).
\end{align*}
To initialize the MOR, the objective is to automatically find reduced dynamics for the state and adjoint equations. In Section~\ref{sec:POD}, we discuss an efficient method for computing a reduced system that significantly reduces the computational complexity while preserving the essential characteristics of the system.

\subsection{The POD method}
\label{sec:POD}

The (discrete) POD method is based on constructing a low-dimensional subspace that can resemble the information carried out by a given set of vectors $\{\bz_j\}_{j=1}^n\subset\mathbb R^m$ (the so-called \emph{snapshots}); cf., e.g., \cite{KV01} and \cite[Section~2.1]{GV17}. Let
\begin{align*}
	\mathscr V=\mathrm{span}\,\big\{\bz_1,\ldots,\bz_n\big\}\subset\mathbb R^m
\end{align*}
be the space spanned by the snapshots with dimension $d_\mathscr V=\dim \mathscr V\le \min\{n,m\}$. To avoid trivial cases we assume $d_\mathscr V\ge0$. For $\ell\le d_\mathscr V$ the POD method generates pairwise orthonormal functions $\{\psi_i\}_{i=1}^\ell$ such that all snapshots can be represented with sufficient accuracy by a linear combination of the $\psi_i$'s. This is done by a minimization of the mean square error between the snapshots and their corresponding $\ell$-th partial Fourier sum:
\begin{equation}
    \label{PODBasisProb}
    \left\{
    \begin{aligned}
    &\min\sum_{j=1}^n\alpha_j \Big| \bz_j-\sum_{i=1}^\ell{(\bz_j,\psi_i)}_W\,\psi_i\Big|_W^2\\
    &\hspace{0.5mm}\text{s.t. } \{\psi_i\}_{i=1}^\ell\subset \mathbb R^m\text{ and }{(\psi_i,\psi_j)}_W=\delta_{ij},~1 \le i,j \le \ell,
    \end{aligned}
    \right.
\end{equation}
where the $\alpha_j$'s are positive weighting parameters for $j=1,\ldots,n$ and for $z,v\in \mathbb R^m$ we have $(z,v)_W\colonequals z^\top Wv$ and  $|z|^2_W\colonequals z^\top Wz$. The symbol $\delta_{ij}$ denotes the Kronecker symbol satisfying $\delta_{ii}=1$ and $\delta_{ij}=0$ for $i\neq j$.

In our application, we have $\bz_j\approx \by(t_j)$ or $\bz_j\approx \bp(t_j)$ for a time grid $\tau_j=(j-1)\Delta\tau$, $j=1,\ldots,n$, with $\Delta\tau=\nicefrac{T}{(n-1)}$, where $\by$ and $\bp$ solve the state equation \eqref{e25a} and the adjoint equation \eqref{eAdj}, respectively. Moreover, we have $W=M$ or $W=S$, and the weights $\alpha_j$ are chosen to resemble a trapezoidal rule for the temporal integration.\hfill\\
An optimal solution to \eqref{PODBasisProb} is denoted as a \emph{POD basis of rank $\ell$}. It can be proven that such a solution is characterized by the eigenvalue problem
\begin{align}
\label{POD_EVP}
\mathcal R\psi_i =\lambda_i\psi_i\quad\text{for }1\le i\le\ell,
\end{align}
where $\lambda_1\geq \ldots \geq \lambda_\ell \geq \ldots \geq \lambda_{n_\mathscr V} > 0$ denote the eigenvalues of the linear, compact, nonnegative and self-adjoint (with respect to the $W$ inner product) operator
\[
\mathcal R:\mathbb R^m \to \mathbb R^m, \quad \mathcal R\psi=\sum_{j=1}^n\alpha_j\,{(\psi,\bz_j)}_W\,\bz_j\quad\text{for }\psi\in \mathbb R^m;
\]
cf., e.g.,  \cite[Lemma~2.2]{GV17}. We refer to \cite[Remark~2.11]{GV17} for an explicit form of the operator $\mathcal R$. Recall that for a solution $\{\psi_i\}_{i=1}^\ell$ to \eqref{POD_EVP} the following approximation error formula holds true:
\[
\sum_{j=1}^n\alpha_j\Big|\bz_j-\sum_{i=1}^\ell{(\bz_j,\psi_i)}_W\,\psi_i\Big|_W^2=\sum_{i=\ell+1}^{d_\mathscr V}\lambda_i;
\]
cf. \cite[Theorem~2.7]{GV17}. We define the POD matrix $\Psi\colonequals [\psi_1|\ldots|\psi_\ell]\in\mathbb R^{m\times\ell}$ and derive a POD-based ROM by Galerkin projection. In particular, if we denote the basis matrices arising from the state snapshots $\by_j$ by $\Psi_y$, the ROM for the state equation can be expressed as 
\begin{align*}
    \dot\by^{\ell_y}(t)=\bm f^\ell_y(t,\by^\ell(t),\bu(t))\colonequals\Psi_y^\top \bm f_y(t,\Psi_y \by^{\ell_y}(t),\bu(t)).
\end{align*}
Here, we consider the approximation $\by(t) \approx\Psi_y \by^{\ell_y}(t)$. Similarly, employing the basis matrix $\Psi_p$ derived from adjoint snapshots $\bp_j$, the ROM for the adjoint equation is obtained through the equation
\begin{align*}
    -\dot\bp^{\ell_p}(t)=\bm f^\ell_p(t,\by^\ell(t),\bu(t))\colonequals\Psi_p^\top \bm f_p(t,\Psi_p \bp^{\ell_p}(t),\Psi_y\by^{\ell_y}(t)).
\end{align*}
For more details we refer the reader to \cite{GV17,HV08,KV01}, for instance.

\subsection{POD-based receding horizon algorithm} 

Now we can explain how the MOR technique based on POD as in \ref{sec:POD} is combined with the receding horizon control Algorithm~\ref{RHA} to stabilize \eqref{e1} around the zero. For a given sampling time $\delta>0$ and a chosen prediction $T>\delta$,  the POD-based RHC approach proceeds through the following steps:
\begin{itemize}
    \item[Step 1.] Compute the optimal control on the first interval $(0,T)$ with FMO discretization based on FE, i.e., \eqref{eq:opt_discrete} with $(\bar{t_0},\bar{\by}_0)\colonequals(0,\by_0)$. Then store all required FE snapshots for the state and adjoint equations and also set the RHC on $(0,\delta)$ equal to this computed optimal control.      
    \item[Step 2.] Generate separate POD bases both for state and for the adjoint equations from the stored snapshots.
    \item [Step 3.] Compute RHC on the interval $(\delta,\infty)$ or, more precisely, on all the subsequent time interval $(k\delta,(k+1)\delta)$ with $k\geq 1$ as follows:
    \begin{itemize}
        \item[(i)] Given an initial pair $(\bar t_0,\bar\by_0)$ with $\bar{t}_0\geq \delta $, approximate the optimal solution to \eqref{eq:opt_discrete}  using the POD-based MOR obtained by the snapshots given in Step 2, to obtain an approximate control for the (unknown) optimal control on the interval $(\bar{t}_0,\bar{t}_0+\delta)$. Due to the low-dimensional nature of the problem, this process is expected to be highly expedient.
        \item[(ii)] Utilize the computed optimal control to steer the state within the high-dimensional FE model on the interval $(\bar{t}_0,\bar{t}_0+\delta)$ to compute the initial vector for the next interval with initial time $(\bar{t_0}+\delta)$. Repeat the process starting from (i).
    \end{itemize} 
\end{itemize}

The entire approach outlined above is outlined in Algorithm~\ref{RRHA_disc}, with the first two steps further summarized in Algorithm~\ref{RRHA_POD}.
\begin{algorithm}[htbp]
\caption{(Computing a POD basis for RHC($\delta,T$))}\label{RRHA_POD}
\begin{algorithmic}[1]
\REQUIRE{The prediction-train horizon $T^\mathsf{train}\ge\delta$, initial state $y_0$, basis tolerance $\mathsf{tol}$, POD weighting matrices $W \in \mathbb R^{m\times m}$ and $D=\text{diag}(\alpha_1,...,\alpha_m) \in \mathbb R^{m\times m}$};\\
\hspace{-6mm}\textbf{Output:} State basis $\Psi_y$, adjoint basis $\Psi_p$, and RHC for the first interval $(0,\delta)$; 
\STATE Find the solution $(\by^*_{T^\mathsf{train}}(\cdot\,;0,\by_0)),\bu^*_{T^\mathsf{train}}(\cdot\,;0,\by_0))$
over the time horizon~$(0,T^\mathsf{train})$ by solving the  open-loop problem
\begin{equation}
    \begin{aligned}
        &\min J^\mathsf{fe}_{T^\mathsf{train}}(\bu;0,\by_0)=\frac{1}{2}\int^{T^\mathsf{train}}_0{|\by(t)|}^2_S+\beta\,{|\bu(t)|}^2_{1}\,\mathrm dt\\
        &\hspace{0.5mm}\text{s.t. }\bu\in L^2(0,T^\mathsf{train};\mathbb{R}^N),\quad\dot\by(t)=\bm f_y(t,\by(t),\bu(t))\text{ for } t\in(0,T^\mathsf{train}),\quad\by(0)=\by_0
    \end{aligned}
\end{equation}
and store all (time discrete) snapshots for state and adjoint equations required in the optimization:
\begin{align}
    Y = [\by^{(1)},...,\by^{(iter)}]\quad \text{and} \quad  P = [\bp^{(1)},...,\bp^{(iter)}].
\end{align}
$\textbf{y}^{(i)}$ is the state solve with corresponding adjoint state $\textbf{p}^{(i)}$ to compute the gradient in the optimization algorithm;
\STATE Compute $\hat Y = W^{1/2}YD^{1/2}$ and $\hat P = W^{1/2}PD^{1/2}$;
\STATE Compute the truncated SVD $\hat Y\overset{tSVD}{=}\Psi_{y} \Sigma_{y} V_{y}^\top$ with truncation value 
\begin{align*}
    \ell_y = \underset{i=1,...,m}{\text{argmax }} \sigma_i \quad \text{s.t.}\quad \sigma_i\leq \mathsf{tol},
\end{align*}
where $\sigma_i$ denote singular values of $\hat Y$, i.e. $\Sigma_{y} = \text{diag}(\sigma_i : i=1,...,m)$;
\STATE Compute the truncated SVD $\hat P\overset{tSVD}{=}\Psi_{p} \Sigma_{p}V_{p}^\top$ with truncation value 
\begin{align*}
    \ell_p = \underset{i=1,...,m}{\text{argmax }} \sigma_i \quad \text{s.t.}\quad \sigma_i\leq \mathsf{tol},
\end{align*}
where $\sigma_i$ denote singular values of $\hat P$, i.e. $\Sigma_{p} = \text{diag}(\sigma_i : i=1,...,m)$;
\end{algorithmic}
\end{algorithm}
Notice that for the snapshot matrices it holds $Y,P\in \mathbb R^{m \times (iter \cdot (1+\nicefrac{T_\infty}{\Delta t}))}$ and for the resulting reduced bases we have $\Psi_y\in \mathbb R^{m \times \ell_y}$ and  $\Psi_p\in \mathbb R^{m \times \ell_p}$ with $\ell_y,\ell_p \ll m$.

\begin{algorithm}[htbp]
\caption{(Reduced RHC($\delta,T$) based on POD)}\label{RRHA_disc}
\begin{algorithmic}[1]
\REQUIRE{The sampling time $\delta$,  the prediction horizon $T\geq \delta$, and the initial state $y_0$;}
\ENSURE{The stability of RHC~$\urh$;}
\STATE Compute POD bases using Algorithm~\ref{RRHA_POD} for $T^{train}\colonequals T$ and get $\Psi_y$, $\Psi_p$, and the pair $(\by_T^*(\cdot\,;0, y_0))$, $\bu^*_T(\cdot\,;0, y_0))$;
\STATE For all $\tau \in [0,\delta)$,  set $\urhl(\tau)\colonequals\bu^*_T(\tau; 0, y_0)$ and $\yrhl(\tau)\colonequals\by^*_T(\tau; 0, y_0)$;
\STATE Set $(\bar{t}_0,\bar{\by}_0)\colonequals(\delta,  \yrhl(\delta))$;
\STATE Find the solution $\bu^*_T(\cdot\,;\bar{t}_0,\bar\by_0)$
over the time horizon $(\bar{t}_0,\bar t_0+T)$ by solving the open-loop problem
\begin{align*}
    \begin{aligned}
        &\min J^\ell_T(\bu;\bar t_0,\bar\by_0)=\frac{1}{2}\int^{\bar{t}_0+T}_{\bar t_0}|\Psi_y\by^\ell(t)|^2_{S}+\beta\,|\bu(t)|^2_{1}\,\mathrm dt\\
        &\text{ s.t. }\bu\in L^2(\bar t_0,\bar t_0+T;\mathbb{R}^N)\text{ and } \begin{cases}
            \dot{\by}^\ell(t)=\bm f_y^\ell(t,\by(t),\bu(t)) &\text{ in } (\bar{t}_0,\bar{t}_0+T),\\
            \by^\ell(\bar t_0)=\Psi_y^\top \bar{\by}_0;
        \end{cases}
    \end{aligned}
\end{align*}
\STATE Compute the optimal state $\by^*_T(\cdot\,;\bar t_0, \bar\by_0)$ over $[\bar{t}_0,\bar{t}_0+\delta)$ solving the FE state equation
\begin{align*}
    \dot{\by}(t)= \bm f_y(t,\by(t),\bu^*_T(t;\bar t_0,\bar\by_0))\text{ in } (\bar t_0,\bar t_0+\delta),\quad\by(\bar t_0)=\bar\by_0;
\end{align*}

\STATE For all $\tau \in [\bar{t}_0,\bar t_0+\delta)$ set $\urh(\tau)=\bu^*_T(\tau; \bar t_0,\bar\by_0)$ and $\byrh(\tau)=\by^*_T(\tau;\bar t_0,\bar\by_0)$;
\STATE Update: $(\bar t_0,\bar\by_0)  \leftarrow (\bar t_0+\delta,\byrh(\bar t_0 +\delta))$;
\STATE Go to Step  4;
\end{algorithmic}
\end{algorithm}

\section{Numerical experiments}
\label{Sec:NE}

In this section, we report on numerical experiments that illustrate the performance of Algorithm \ref{RRHA_disc} in comparison with Algorithm \ref{RHA} for the (high-dimensional) FE model. We have employed both algorithms for the stabilization of an exponentially unstable parabolic equation.

We consider various values of the prediction horizon $T$ while maintaining a constant sampling time $\delta=0.25$. Throughout the experiments, we fix $T_{\infty}=10$ as the final computation time. For solving the finite horizon optimal control problems with the $\ell_1$-norm, we apply the forward-backward splitting algorithm since the associated proximal operator can be evaluated efficiently. We applied a proximal gradient method as those  investigated in\cite{azmi2023nonmonotone,MR2792408,MR2678081,MR2650165} on the convex composite problem \eqref{e94}. More precisely, we followed the iteration rule
\begin{equation*}
    \bu^{j+1}=\prox_{\alpha_j\mathcal G}\big(\bu^j-\alpha_j D\mathcal F_T^{\bar t_0,\bar y_0}(\bu^j)\big)=\prox_{\alpha_j\mathcal G}(\bu^j-\alpha_j\mathcal B^\star p^{j}),
\end{equation*}
where $p^{j}\colonequals p(y^{j})$ is the solution of \eqref{e52} for the forcing function $\Delta y^j$ instead of $\Delta y^*$, and  $y^{j}=y(\bu^{j})$ is defined as the solution of \eqref{e41} for the control $\mathbf{u}^{j}$ instead of $\bu$. Moreover, the stepsize $\alpha_j$ is computed by a non-monotone linesearch algorithm which uses the Barzilai-Borwein stepsizes \cite{AzmiKunisch3,AK22a,BB88}  corresponding to the smooth part $ \mathcal F_T^{\bar t_0,\bar y_0}$  as the initial trial stepsize, see \cite{azmi2023nonmonotone,MR2792408,MR2678081,MR2650165} for more details. In this case the optimization algorithm was terminated as the following condition held
\begin{equation*}
    \frac{\|\bu^{j+1}-\bu^{j}\|_{\U_T(t_k)}}{\|\bu^{j+1}\|_{\U_T(t_k)}}\leq 10^{-4}.
\end{equation*}
The evaluation of the proximal operator $\prox_{\bar{\alpha}\mathcal{G}}$ was carried out by pointwise evaluation of \eqref{e87} at time grid points. Further, at every time grid point,  $\prox_{\bar{\alpha} g }$ was computed by \eqref{e85}, where the zero $\mu^*$ of the function $\psi(\mu)$ defined in \eqref{e95} was computed by the bisection method with the tolerance $10^{-10}$.
In our numerical tests, the spatial domain is defined as $\Omega\colonequals(0,1)^2 \subset \mathbb{R}^2$ and Figure~\ref{Fig:1} depicts the control domain $\omega$ as the union of 13 open rectangles. The control domain consists of thirteen percent of the domain.
\begin{figure}[htb!]
    \centering
    \includegraphics[height=4cm,width=4cm]{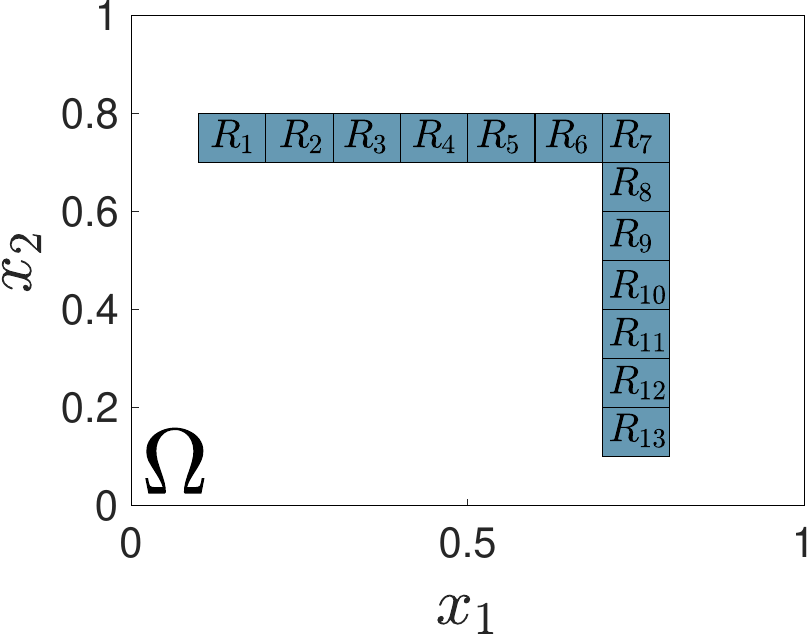}
    \caption{Spatial domain $\Omega$ and sub-rectangles $R_1,\ldots,R_{13}$ for the control actuators.\label{Fig:1}}
\end{figure}
The spatial discretization was done using a conforming linear FE scheme, employing continuous piecewise linear basis functions over a uniform triangulation with a diameter $h = 0.0442$ (resulting in $m=961$ interior nodes). Subsequently, the semi-discrete system of ordinary differential equations resulting from spatial discretization was numerically solved using the Crank-Nicolson time-stepping method with a step-size $\Delta t=\nicefrac{1}{80}$. In this context, the resulting discretized system achieves second-order accuracy in time and can equivalently be interpreted as a system discretized in time through a Petrov-Galerkin scheme. This scheme is based on continuous piecewise linear basis functions for the trial space and piecewise constant test functions.
Throughout our numerical simulation, we set $\nu=0.1$ and choose for $x=(x_1,x_2)\in \mathbb R^2$ that
\begin{align*}
    a(t,x)\colonequals -2-0.8\,|\sin(t+x_1)|, \quad b(t,x)\colonequals\left(
    \begin{array}{c}
          0.1\cos(t)-0.01(x_1+x_2)\\
          0.2x_1x_2\cos(t)
     \end{array}
     \right),
\end{align*}
and $y_0(x)\colonequals3\sin(\pi x_1) \sin(\pi x_2)$. For this choice, the uncontrolled state $\yun$ is exponentially unstable. Moreover, we have $\|\bm y_\mathsf{un}\|_{L^2(0,T_\infty;\mathbb R^m)}=3.99 \times 10^2$, and $|\bm y_\mathsf{un}(T_\infty)|_M=4.00 \cdot 10^2$.  The regularization parameter, denoted by $\beta$, is assigned a value of $5$.

Figure~\ref{Fig:2} illustrates the evolution of $\log(|\bm y_\mathsf{rh}(t)|_M)$ (solid line) alongside the corresponding reduced-order model $\log(|\bm y^\ell_\mathsf{rh}(t)|_M)$.
\begin{figure}[htb!]
    \centering
    \includegraphics[trim = 0mm 0mm 0mm 0mm, width=0.6\textwidth]{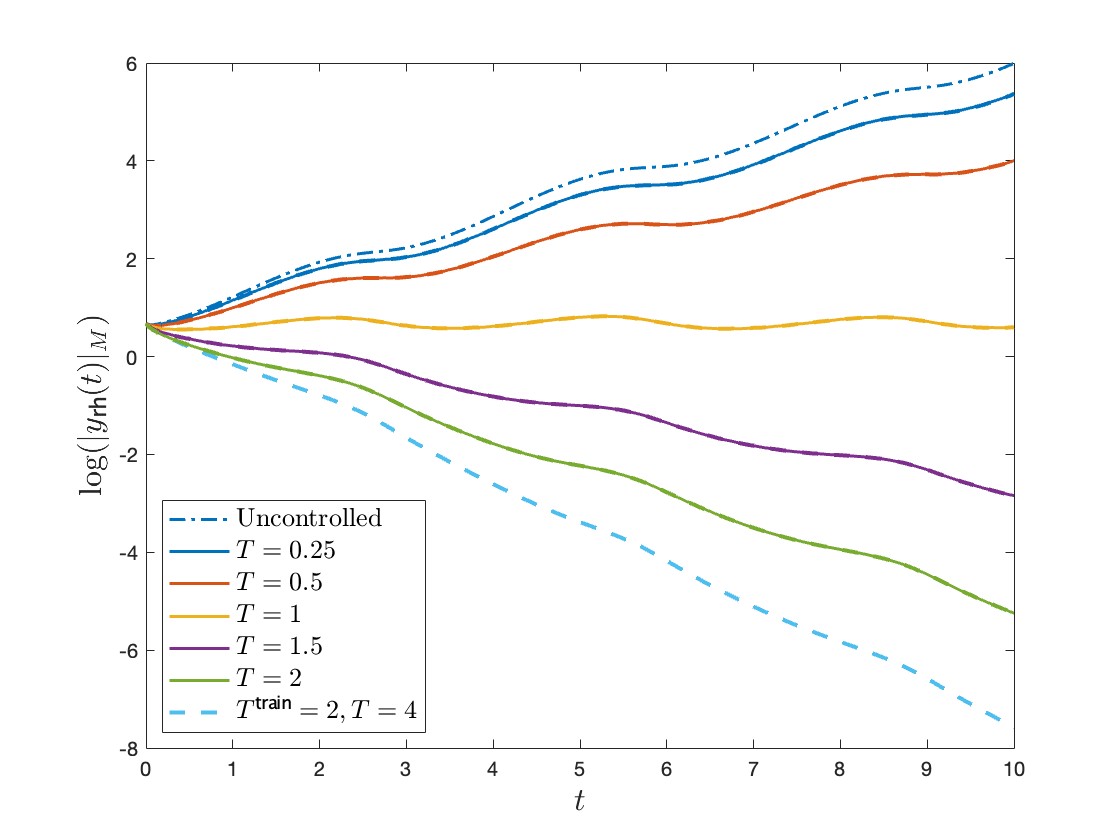}
    \caption{Evolution of $\log(|\bm y_\mathsf{rh}(t)|_M)$ (solid line) along with the corresponding reduced-order model $\log(|\yrhl(t)|_M)$ (dashed line) for various choices of $T$. For the reduced models, $T^\mathsf{train}$ is equal to $T$ in the first five cases. The finite element model for $T=4$ is numerically extremely expensive, and for this reason, it is not computed in this article.}
    \label{Fig:2}
\end{figure}
The corresponding results are gathered in Table~\ref{table1}. For the reduced models, $T^\mathsf{train}$ aligns with $T$ in the initial five cases, i.e., $T \in \{0.25,0.5,1,1.5,2\}$. Notably, in this plot, the reduced model closely aligns with the full model, demonstrating their behavior.
\begin{table}[htbp]
    \begin{center}
        \scalebox{0.9}{
        \begin{tabular}{ | c | c | c | c | c | c | c |}\hline
            Model &$T^\mathsf{train}$& $T$  & $\|\bm y_\mathsf{rh}\|_{L^2(0,T_{\infty};\mathbb R^m)}$  & $J_{T_\infty}^\mathsf{fe}$ &$|\bm y_\mathsf{rh}(T_{\infty})|_M$ & CPU-time\\\hline
            FE & - &$ 0.25$ &$ 2.30 \times 10^{2}$ & $ 5.31 \times 10^{5}$& $ 2.15 \times 10^{2}$& $3.16 \times 10^{2}$s \\ \hline
            FE & - &$ 0.5$ &$ 7.04 \times 10^{1}$ & $ 5.66 \times 10^{4}$& $ 5.47 \times 10^{1}$& $7.50 \times 10^{2}$s\\ \hline
            FE & - &$ 1$ &$ 6.27 \times 10^{0}$ & $ 2.60 \times 10^{3}$& $ 1.82 \times 10^{0}$& $1.73 \times 10^{3}$s \\ \hline
            FE & - &$ 1.5$ &$ 2.24 \times 10^{0}$ & $ 1.03 \times 10^{3}$& $ 5.83 \times 10^{-2}$& $3.12 \times 10^{3}$s \\ \hline
            FE & - &$ 2$ &$ 1.68 \times 10^{0}$ & $ 7.77 \times 10^{2}$& $ 5.30 \times 10^{-3}$& $5.97 \times 10^{3}$s \\ \hline\hline
            POD & $0.25$ &$ 0.25$ &$ 2.29 \times 10^{2}$ & $ 5.26 \times 10^{5}$& $ 2.14 \times 10^{2}$& $2.08 \times 10^{1}$s \\ \hline
            POD & $0.5$ &$ 0.5$ &$ 7.03 \times 10^{1}$s & $ 5.65 \times 10^{4}$& $ 5.46 \times 10^{1}$& $3.35 \times 10^{1}s$ \\ \hline
            POD & $1$ &$ 1$ &$ 6.26 \times 10^{0}$ & $ 2.60 \times 10^{3}$& $ 1.82 \times 10^{0}$& $6.63 \times 10^{1}$s \\ \hline
            POD & $1.5$ &$ 1.5$ &$ 2.24 \times 10^{0}$ & $ 1.03 \times 10^{3}$& $ 5.83 \times 10^{-2}$& $1.26 \times 10^{2}$s \\ \hline
            POD & $2$ &$ 2$ &$ 1.68 \times 10^{0}$ & $ 7.77 \times 10^{2}$& $ 5.30 \times 10^{-3}$& $1.79 \times 10^{2}$s \\ \hline
            POD & $2$ &$ 4$ &$ 1.72 \times 10^{0}$ & $ 6.81 \times 10^{2}$& $ 5.14 \times 10^{-4}$& $1.85 \times 10^{2}$s \\ \hline
        \end{tabular}}
    \end{center}
    \caption{Numerical results for the FE and reduced POD RHC framework.}
    \label{table1}
\end{table}
Further insights into the absolute $L^2$-error, depicted by $t \mapsto |\bm y_\mathsf{rh}(t)-\bm y_\mathsf{rh}^\ell(t)|_M$, are provided in Figure \ref{Fig:3}. Notably, as the training duration increases, the results improve, showcasing a trend toward capturing the entire time-varying periodic dynamics for the ROM. The longer the training phase, the more information can be reflected in the bases. However, this is also a trade-off because as the training period increases (first RHC iteration,i.e. $T^\mathsf{train}$ increases), the numerical costs also rise.

We have also observed once a sufficiently good basis has been computed (for example, for $T^\mathsf{train}=2$), one can extend the prediction horizon $T$ for the reduced model. Consequently, in a much shorter computation time, superior results can be achieved compared to the full model. This suggests that it is more beneficial to look further into the future and obtain an approximate solution, rather than calculating the exact, more computationally expensive solution for a shorter prediction horizon.
  \begin{figure}[htb!]
    \centering
    \includegraphics[trim = 0mm 0mm 0mm 0mm, width=0.6\textwidth]{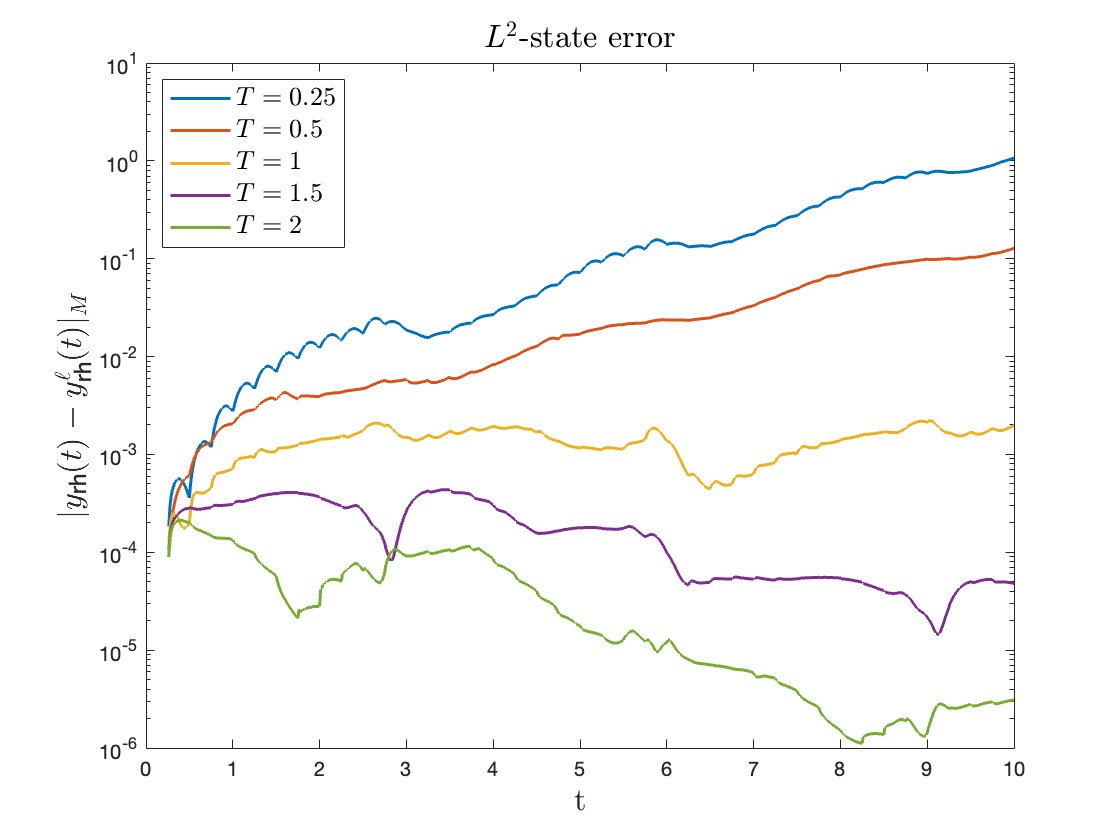}
    \caption{Evolution of the absolute $L^2$-state error $|(\bm y_\mathsf{rh})_i(t)-(\yrhl)_i(t)|_M$ for various choices of $T$.}
\label{Fig:3}
\end{figure}

\noindent
In Figure~\ref{Fig:4a} the mapping $t\mapsto|(\urh)_i(t)|$, $i=1,\ldots,13$, is plotted.
\begin{figure}[htb!]
    \centering
    \includegraphics[trim = 0mm 0mm 0mm 0mm, width=0.9\textwidth]{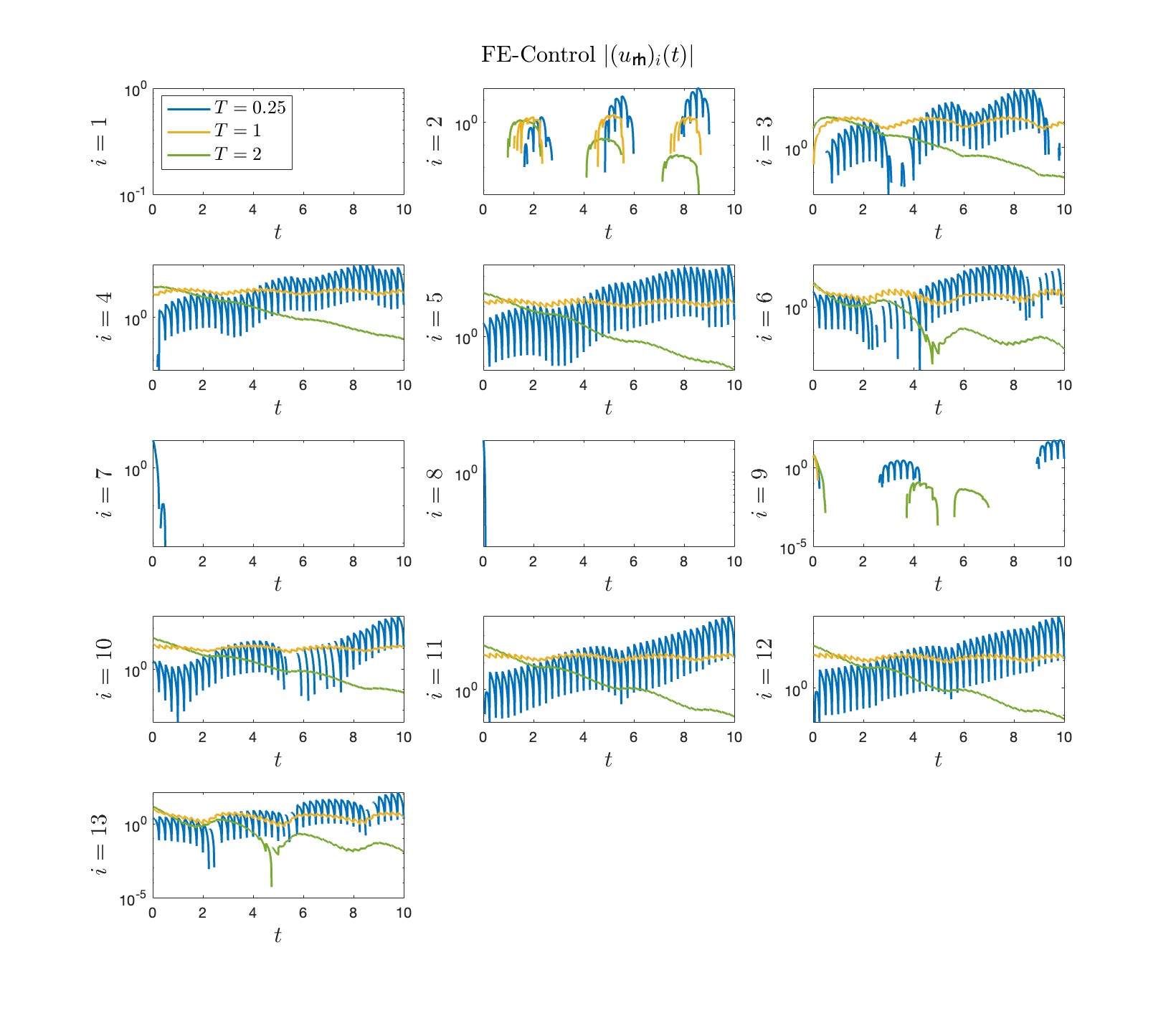}
    \caption{Evolution of the absolute FOM-based control error $t\mapsto|(\urh)_i(t)|$ for $i=1,...,13$.}
    \label{Fig:4a}
\end{figure}
The sparsity of the FOM-based control can be observe numerically: $|(\urh)_1|=0$ and $|(\urh)_1|\approx0$ for $i=7$ and $8$ in $[0,T]$. Furthermore, $|(\urh)_i|$, $i=2$, $3$ and $9$ is turned on and of in $[0,T]$

In Figure~\ref{Fig:4}, the error in the control is illustrated.
\begin{figure}[htb!]
    \centering
    \includegraphics[trim = 0mm 0mm 0mm 0mm, width=0.9\textwidth]{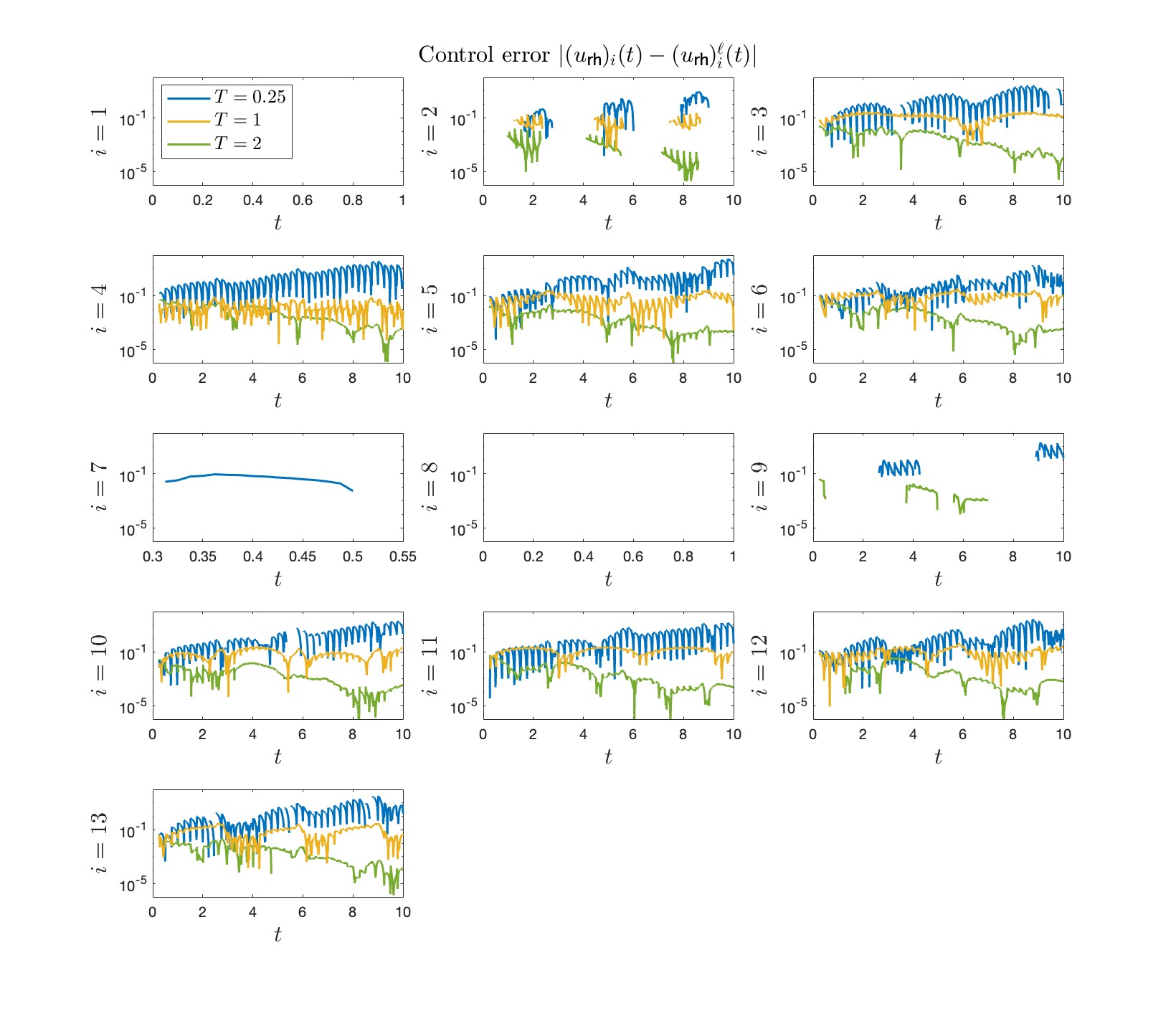}
    \caption{Evolution of the absolute control error $t\mapsto|(\urh)_i(t)-(\urhl)_i(t)|$ for $i=1,...,13$. It turns out that the sparsity of the FOM-based RHC control $\urh$ is very well aproximated by the POD-based RHC control $\urhl$.}
    \label{Fig:4}
\end{figure}
The utilization of the $\ell_1$-norm for the control vector contributes to increased sparsity in the active control compartments. This enhances stabilization with a small number of active controllers. It is noteworthy that the sparsity pattern aligns remarkably well between the full and reduced models, indicating that whenever a control is zero in the full model, the corresponding control in the reduced model is also zero. This consistency emphasizes the effectiveness of the reduced-order model in capturing the sparsity characteristics of the original full model.

Regarding the CPU time, the following observations can be made from Table \ref{table2}: It is noticeable that the first iteration takes approximately the same amount of time for both the reduced and the normal RHC algorithms. Once the basis is built, the reduced model becomes significantly faster. 
This efficiency gain is particularly advantageous in MPC scenarios, where problems with infinite time horizons are commonly considered, i.e. $T_\infty \to \infty$.

In each iteration, the reduced model consistently exhibits a substantial speedup compared to its full counterpart. This speedup becomes increasingly pronounced as the time horizon approaches infinity. The reduced model's ability to capture essential information from the training phase translates into improved computational efficiency, making it an attractive choice for real-time applications and scenarios where rapid decision-making is crucial.

\begin{table}[htbp]
\begin{center}
\scalebox{0.9}{
  \begin{tabular}{ | c | c | c | c | c  |}
    \hline
    Model &$T^\mathsf{train}$& $T$  &1st RHC iteration  & $i$-th RHC iteration \\
    \hline
    FE & - &$ 0.25$ &\phantom{11}7s & \phantom{11}7s \\ \hline
    FE & - &$ 0.5$ &\phantom{1}19s &\phantom{1}19s \\ \hline
    FE & - &$ 1$ &\phantom{1}51s &\phantom{1}42s \\ \hline
    FE & - &$ 1.5$ &110s & \phantom{1}80s \\ \hline
    FE & - &$ 2$ &165s & 130s \\ \hline \hline
     POD & $0.25$ &$ 0.25$ &\phantom{11}7s & \phantom{1}0.3s \\ \hline
    POD & $0.5$ &$ 0.5$ &\phantom{1}19s &\phantom{1}0.3s \\ \hline
    POD & $1$ &$ 1$ &\phantom{1}52s &\phantom{1}0.4s \\ \hline
    POD & $1.5$ &$ 1.5$ &111s & \phantom{1}0.4s \\ \hline
    POD & $2$ &$ 2$ &166s & \phantom{1}0.4s \\ \hline
    POD & $2$ &$ 4$ &167s & \phantom{1}0.5s \\ \hline 
 \end{tabular}}
 \end{center}
 \caption{Average CPU-time (in seconds) for the first RHC iteration and the $i$-th RHC iteration with $i=2,...,40$.}
 \label{table2}
\end{table}

  \begin{figure}[htb!]
    \centering
     \label{Fig:5}
        \includegraphics[trim = 0mm 0mm 0mm 0mm, width=0.6\textwidth]{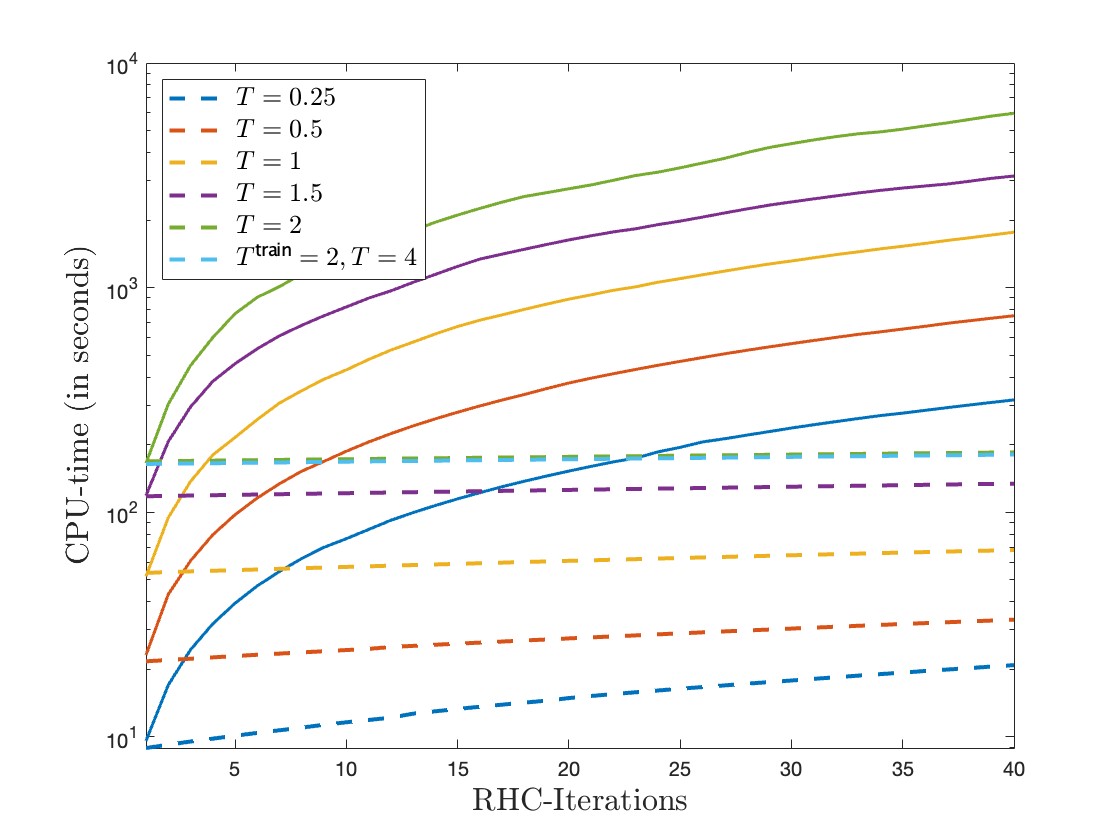}
     \caption{Comparison of CPU times between the FOM and ROM. The solid line \texttt{(-)} represents the CPU time of the FOM, and the dashed line \texttt{(--)} represents the CPU time for the ROM for various choices of $T$. For the reduced models, $T^\mathsf{train}$ is equal to $T$ in the first five cases.}
\end{figure}

\section{Discussion}

\textbf{Stabilization and Computational Efficiency with ROM RHC:}\\
The implementation of a reduced-order model within the receding horizon control framework demonstrates a dual advantage of stabilizing the system equation while significantly improving computational efficiency. The ability of ROMs to offer both stability and enhanced computational efficiency makes them a promising choice for real-time applications, where the speed of computation is of utmost importance. Importantly, this discussion opens up the possibility of extending the ROM framework to nonlinear dynamics, further broadening its potential applications. The ability of ROMs to handle non-linear systems would be a substantial advancement, providing a more versatile and comprehensive solution for control in complex dynamic environments.
\\
\\
\textbf{Sparsity Pattern Alignment:}\\
The alignment of sparsity patterns between the full system and the reduced model is a significant observation. The remarkable consistency in sparsity patterns indicates that the reduced model faithfully represents the system dynamics. This alignment is crucial in ensuring that the reduction process retains the essential features of the system, affirming the reliability of the ROM in capturing system behavior while maintaining computational efficiency.
\\
\\
\textbf{Impact of Training Duration on ROM Performance:}\\
A noteworthy observation is the correlation between the training duration and the performance of ROM in capturing time-varying periodic dynamics. As the training phase duration increases, there is a clear trend towards improved results, indicating a more comprehensive representation of system dynamics in the ROM bases. This emphasizes the importance of allowing sufficient time for the training phase to enable the ROM to capture a wide range of information.\\
An intriguing question arises concerning the adaptivity of the training phase to capture all important information efficiently. The suggestion of an adaptive training strategy is explored to align the training phase with the dynamic characteristics of the system. This adaptive approach not only ensures the inclusion of critical dynamics but also accelerates the initial RHC iteration. This innovation holds the potential for optimizing the trade-off between accuracy and computational efficiency.
\\
\\
In conclusion, the discussion highlights the promising aspects of ROM-based RHC, emphasizing the need for a balanced approach to training duration, and accuracy in capturing system dynamics. The observed alignment in sparsity patterns further reinforces the credibility of reduced models in practical applications. The ongoing exploration of adaptive training strategies holds the potential to further enhance the efficiency and effectiveness of ROM-based receding horizon control.
\bibliographystyle{plain}
\bibliography{References.bib}

\end{document}